\documentclass[11pt]{article}
\usepackage{amsmath}
\usepackage{amsthm}
\usepackage{latexsym}
\usepackage{graphicx, epsfig}
\usepackage{amssymb,amsmath}  
\usepackage{subfigure}
\usepackage{hyperref}
\usepackage{mathrsfs}
\usepackage{mathdots}
\usepackage{cases}

\usepackage{pgfpages,tikz,tikz-cd,pgfkeys,pgfplots}
\usetikzlibrary{arrows,positioning,matrix,fit,backgrounds,shapes}

\usetikzlibrary{fit,matrix}
\tikzset{
mN/.style = {
    draw=#1, semithick, inner sep=0pt}
             }

\definecolor{LemonChiffon}{rgb}{100, 98, 80}
\definecolor{myblue}{rgb}{0,0.4,0.8}
\definecolor{orange}{rgb}{1, 0.4, 0}
\definecolor{mygreen}{rgb}{0, 0.8, 0}
\definecolor{myred}{rgb}{204, 0, 0}
\definecolor{violet}{RGB}{0.4,0.2,1}
\definecolor{brown}{rgb}{0.6, 0.4, 0}

\makeatletter
\newtheorem*{rep@XThm}{\rep@title}
\newcommand{\newreptheorem}[2]{%
\newenvironment{rep#1}[1]{%
 \def\rep@title{#2 \ref{##1}}%
 \begin{rep@XThm}}%
 {\end{rep@XThm}}}
\makeatother

\newtheorem{theorem}{Theorem}[section]

\newreptheorem{XThm}{Theorem}
\newtheorem{lemma}[theorem]{Lemma}
\newtheorem{proposition}[theorem]{Proposition}
\newtheorem{corollary}[theorem]{Corollary}

\theoremstyle{definition}
\newtheorem{definition}[theorem]{Definition}

\newtheorem{example}[theorem]{Example}

\newcounter{statement}
\makeatletter
\newcommand{\leqnomode}{\tagsleft@true\let\veqno\@@leqno}
\newcommand{\reqnomode}{\tagsleft@false\let\veqno\@@eqno}
\makeatother
\newcommand{\statement}[2]{%
   \begin{equation}\refstepcounter{statement}%
      \leqnomode%
      \tag{S\thestatement}\label{stm:#1}%
      \parbox{\dimexpr\linewidth-4em}{#2}%
   \end{equation}%
}

\title{Matrix periods and competition periods of Boolean Toeplitz matrices II}
\date{}
\author{Gi-Sang Cheon$^{a, b}$, Bumtle Kang$^{b}$, Suh-Ryung Kim$^{b, c}$, and Homoon Ryu$^{b,c}$ \\
{\footnotesize $^a$ \textit{Department of Mathematics, Sungkyunkwan
University, Suwon 16419, Rep. of Korea}}\\
{\footnotesize $^{b}$ \textit{Applied Algebra and Optimization
Research Center, Sungkyunkwan University,}}\\{\footnotesize\textit{
Suwon 16419, Rep. of Korea}}\\
 {\footnotesize $^{c}$ \textit{Department of Mathematics Education,
Seoul National University,}}\\{\footnotesize\textit{
Seoul 08826, Rep. of Korea}}\\
{\footnotesize gscheon@skku.edu, lokbt@hotmail.com, srkim@snu.ac.kr, and ryuhomun@naver.com}}

\begin{document}

\maketitle

\begin{abstract}
This paper is a follow-up to the paper [Matrix periods and competition periods of Boolean Toeplitz matrices, {\it Linear Algebra Appl.} 672:228--250, (2023)].
Given subsets $S$ and $T$ of $\{1,\ldots,n-1\}$, an $n\times n$ Toeplitz matrix $A=T_n\langle S ; T \rangle$ is defined to have $1$ as the $(i,j)$-entry if and only if $j-i \in S$ or $i-j \in T$.
In the previous paper, we have shown that the matrix period and the competition period of Toeplitz matrices $A=T_n\langle S; T \rangle$ satisfying the condition ($\star$) $\max S+\min T \le n$ and $\min S+\max T \le n$ are $d^+/d$ and $1$, respectively, where $d^+= \gcd (s+t \mid s \in S, t \in T)$ and $d = \gcd(d, \min S)$.
In this paper, we claim that even if ($\star$) is relaxed to the existence of elements $s \in S$ and $t \in T$ satisfying $s+t \le n$ and $\gcd(s,t)=1$, the same result holds.
There are infinitely many Toeplitz matrices that do not satisfy ($\star$) but the relaxed condition.
For example, for any positive integers $k, n$ with $2k+1 \le n$, it is easy to see that $T_n\langle k, n-k;k+1, n-k-1  \rangle$ does not satisfies ($\star$) but satisfies the relaxed condition.
Furthermore, we show that the limit of the matrix sequence $\{A^m(A^T)^m\}_{m=1}^\infty$ is $T_n\langle d^+,2d^+, \ldots, \lfloor n/d^+\rfloor d^+\rangle$.
\end{abstract}

\section{Introduction}
A {\em binary Boolean ring} $(\mathbb{B}, +, \cdot)$ is a set $\mathbb{B} = \{0,1\}$ with two binary operations $+$ and $\cdot$ on $\mathbb{B}$ defined by
\[
\begin{array}{c|cc}
+ & 0 & 1 \\
\hline
0 & 0 & 1 \\
1 & 1 & 1
\end{array}  \quad \text{ and }\quad
\begin{array}{c|cc}
\cdot & 0 & 1 \\
\hline
0 & 0 & 0 \\
1 & 0 & 1
\end{array}
\]

Let $\mathbb{B}_n$ be the set of $n \times n$ {\it Boolean matrices} with entries from a binary Boolean ring.
Take $A \in \mathbb{B}_n$.
The {\em matrix period} of $A$ is the
smallest positive integer $p$ for which there is a positive integer $M$ such that $A^m = A^{m+p}$ for any integer $m \ge M$.
We note that the rows $i$ and $j$ of $A^m$ have a common nonzero entry in some column if and only if the $(i,j)$-entry of $A^m(A^T)^m$ is $1$.
Consider the matrix sequence $\{A^m(A^T)^m\}_{m=1}^\infty$.
Since $|\mathbb{B}_n|= 2^{n^2}$, there is the smallest positive integer $q$ such that 
\[A^{q+i}(A^T)^{q+i}=A^{q+r+i}(A^T)^{q+r+i}\] for some positive integer $r$ and every nonnegative integer $i$.
Then there is also the smallest positive integer $p$ such that $A^{q}(A^T)^q=A^{q+p}(A^T)^{q+p}$.
Those integers $q$ and $p$ are called the {\it competition index} and {\it competition period} of $A$, respectively, which was introduced by Cho and Kim~\cite{cho2013competition}.
Refer to \cite{cho2011competition,kim2008competition,kim2010generalized, kim2015characterization, kim2012bound} for further results on competition indices and competition periods of digraphs.

A (0,1)-matrix $A=(a_{ij})\in \mathbb{B}_n$ is called a {\em Boolean Toeplitz matrix} if $a_{ij}=a_{j-i} \in \mathbb{B}= \{0,1\}$, {\it i.e.} $A$ is of the Toeplitz form:
\[
\begin{bmatrix}
a_0 & a_1 & \cdots & a_{n-1} \\
a_{-1} & a_0  & \ddots &\vdots \\
\vdots & \ddots & \ddots  & a_1  \\
a_{-n+1}&\cdots &a_{-1} & a_0
\end{bmatrix}.
\]
Accordingly, a Boolean Toeplitz matrix $A\in \mathbb{B}_n$ is determined by two nonempty subsets $S$ and $T$, not necessarily disjoint, of $[n-1]:=\{1,\ldots, n-1\}$ so that $a_{ij}=1$ if and only if $j-i \in S$ or $i-j \in T$.
We assume that $S=\{s_1,\ldots,s_{k_1}\}$ and $T=\{t_1,\ldots,t_{k_2}\}$ where
$$
1\le s_1<\ldots<s_{k_1}<n\quad{\rm and}\quad 1\le t_1<\ldots<t_{k_2}<n.
$$
Note that $S=\{j\mid a_j=1\}$ and $T=\{i\mid a_{-i}=1\}$.
In this context, we denote a Boolean Toeplitz matrix $A$ associated with index sets $S$ and $T$ by $T_n\langle s_1,\ldots,s_{k_1};t_1,\ldots,t_{k_2}\rangle$ or simply by $T_n\langle S;T\rangle$.
Accordingly,
\[
s_1 = \min S, \quad s_{k_1} = \max S,\quad t_1 = \min T, \quad \text{and} \quad t_{k_2} = \max T.
\]

Given two finite subsets $S$ and $T$ of $\mathbb{Z}$, we denote $\gcd (S+T)$ and $\gcd(S\cup T)$ by
\[
\gcd (s+t \mid s \in S, t\in T) \quad \text{and} \quad \gcd (r \mid r \in S \cup T),
\]respectively.

Given a Boolean square matrix $A$, we call the digraph with adjacency matrix $A$ {\it the digraph of $A$} and denote it by $D(A)$.
In a digraph, if there is a directed walk $W$ from a vertex $u$ to a vertex $v$, then we write \[
u \xrightarrow{W} v,
\]
and if $W$ has length $1$, then $(u,v)$ is an arc and we simply write
\[
u \rightarrow v.
\]

Let $A= T_n \langle S; T \rangle$ be a Toeplitz matrix.
Cheon et al. \cite{exp} gave a necessary condition for $D(A)$ having a directed walk of a specific length from a vertex to a vertex as follows.

\begin{lemma}[\cite{exp}]\label{lem:cwcorr} Let $W$ be a $(u,v)$-directed walk of length $m$ in $D(A)$. Then there are nonnegative integer sequences
$(a_{i})_{i=1}^{k_1}$ and $(b_i)_{i=1}^{k_2}$ such that
\begin{eqnarray*}
 \sum_{i=1}^{k_1} a_i s_i - \sum_{i=1}^{k_2} b_i t_i = v-u \quad
\mbox{and} \quad m = \sum_{i=1}^{k_1}a_i + \sum_{i=1}^{k_2}
b_i.
\end{eqnarray*}\end{lemma}

They also gave a sufficient condition for $D(A)$ having a directed walk from a certain vertex to a certain vertex in the following way.
Yet, this sufficient condition possesses a rather strong condition $s_{k_1}+t_{k_2} \le n$.

\begin{lemma}[\cite{exp}]\label{lem:conv}
Let $A= T_n \langle S; T \rangle$ be a Toeplitz matrix with $s_{k_1}+t_{k_2} \le n$.
Suppose that there exist two integers $u,v \in [n]$ and some nonnegative integers $a_s$ and $b_t$ for $s \in S$ and $t\in T$ satisfying
\[
\sum_{s \in S} a_s s - \sum_{t \in T} b_t t = v-u.
\]
Then there exists a directed walk from $u$ to $v$ in $D(A)$.
\end{lemma}

Subsequently, the authors \cite{period} improved Lemma~\ref{lem:conv} by substituting the condition $s_{k_1}+t_{k_2}  \le n$ with 
\begin{itemize}
\item[($\star$)] $s_1+t_{k_2} \le n$ and $s_{k_1}+t_1 \le n$,
\end{itemize}
albeit only guaranteeing it for sufficiently long walks by Theorem~\ref{lem:pqr}.
In \cite{period}, the periods of Toeplitz matrices satisfying ($\star$) are given. 
In this paper, we extend the Toeplitz matrix family whose periods can be computed from the family of Toeplitz matrices satisfying ($\star$) to the family of ``walk-ensured'' Toeplitz matrices as follows. 
In \cite{period}, it is shown that the family of Toeplitz matrices satisfying ($\star$) is contained in the family of walk-ensured Toeplitz matrices. 

\begin{theorem}[\cite{period}]\label{thm:minmax}
Let $A= T_n \langle S; T \rangle$ be a Toeplitz matrix.
If ($\star$) holds, then $A$ is a walk-ensured Toeplitz matrix.
\end{theorem}

We now present the definition of walk-ensured Toeplitz matrices. 

\begin{definition}
A Toeplitz matrix $A = T_n \langle S ; T \rangle$ is {\it walk-ensured} if there exists a positive integer $M$ such that whenever a pair of vertices $u$ and $v$ in $D(A)$ satisfies $v-u \equiv \ell s_1 \pmod {d^+}$ for an integer $\ell \ge M$ and $d^+ = \gcd(S+T)$, there exists a directed $(u,v)$-walk of length $\ell$ in $D(A)$. 
\end{definition}

For instance, a primitive Toeplitz matrix is obviously walk-ensured.
In general, determining whether or not a Toeplitz matrix is walk-ensured is not easy.

In this paper, we provide periods of walk-ensured Toeplitz matrices (Theorem~\ref{thm:walk-ensuredpr}) and methods for calculating the period of a Toeplitz matrix satisfying specific conditions even if it is not necessarily walk-ensured (Theorems~\ref{cor:gcdsame}, \ref{thm:snksrc}, and \ref{thm:main3}). 

Though we have found a way to determine periods of walk-ensured Toeplitz matrices, but actually identifying walk-ensured Toeplitz matrices is not easy as we mentioned above. 
In this context, we propose a method to find a parameter to be added to $S$ or $T$ that preserves walk-ensuredness when a walk-ensured Toeplitz matrix $T_n\langle S; T \rangle$ is given (Theorem~\ref{thm:1}). 
By using this method, we present two large families of walk-ensured Toeplitz matrices (Theorems~\ref{thm:st} and \ref{thm:main1}).

Sections~\ref{sec:pre} and \ref{sec:cont} prepare the ground work to derive our main results given in Sections~\ref{sec:pfthm2} and \ref{sec:pfthm1} stated above. 
Section~\ref{sec:cont} especially introduces contractions of digraphs derived from residue classes to gain a concise overview of digraph structure, and investigates the relationship between the digraph $D$ of a walk-ensured Toeplitz matrix and the contraction of $D$.

\section{Preliminaries}\label{sec:pre}

The following sets were introduced in \cite{period} to deal with directed walks in the digraph of a Toeplitz matrix.
For convenience, we mean by a {\it directed walk of a matrix} $A$ a directed walk of the digraph of $A$. 

\begin{definition}[\cite{period}]\label{def:pqr}
For a Toeplitz matrix $A=T_n\langle S;T\rangle$ and a positive integer $i$, we introduce the following sets:
\begin{itemize}
\item $P_i(A) = \{\ell \in \mathcal{I}_n \mid \ell \equiv is_1 \pmod {d^+}\}$ where $d^+ = \gcd(S+T)$;
\item $Q_i(A) = \{\sum_{j=1}^{k_1} a_js_j - \sum_{j=1}^{k_2} b_j t_j \in \mathcal{I}_n\mid a_j,b_j \in \mathbb{Z}^+_0, \sum_{j=1}^{k_1} a_j + \sum_{j=1}^{k_2} b_j = i\}$;
\item $R_i(A)$ is the set of $\ell \in\mathcal{I}_n$ such that, for any vertices $u$ and $v$ with $v-u = \ell$, there exists a directed $(u,v)$-walk of length $i$ in $D(A)$
\end{itemize}
where $\mathcal{I}_n = [-n+1, n-1]$. 
\end{definition}

For example, let $A = T_6 \langle 2, 4 ; 5 \rangle$ and $i = 2$.  
Then $\gcd (S+T) = 1$ and so $P_2(A) = \{-5, -4, \ldots, 4, 5\}$. 
One may check that $Q_2(A) = \{ 2\times 2, 2-5, 4-5\} = \{-3, -1, 4\}$.
There are no walks from $5$ to $2$ and from $3$ to $2$, so $R_2(A) = \{4\}$.

In general, by Lemma~\ref{lem:cwcorr} and the following proposition, we see that
\begin{equation}\label{eq:pqr}
R_i(A) \subseteq Q_i(A) \subseteq P_i(A)
\end{equation}
 for any Toeplitz matrix $A$ and a positive integer $i$.

\begin{proposition}[\cite{period}]\label{lem:xiperiod}
	For nonempty sets $S, T \subseteq [n-1]$, let $d = {\rm gcd}(S+T)$.
	Then for any integers $a_i$, $b_j$,
	\[\sum_{i=1}^{k_1}a_is_i - \sum_{i=1}^{k_2}b_it_i \equiv \left(\sum_{i=1}^{k_1}a_i+\sum_{i=1}^{k_2}b_i\right) s_1 \pmod d.\]
\end{proposition}

The reverse direction of containment in \eqref{eq:pqr} holds under certain conditions.

\begin{theorem}[\cite{period}]\label{lem:pqr}
For a Toeplitz matrix $A = T_n \langle S; T \rangle$ with $\min S + \max T \le n$ and $\max S + \min T \le n$, there exists a positive integer $M$ such that $P_i(A) = Q_i(A) = R_i(A)$ for any $i \ge M$.
\end{theorem}

We generally refer to a Toeplitz matrix with such an $M$ in the above theorem as a ``walk-ensured" Toeplitz matrix.

\begin{proposition}\label{prop:walk-ensured}
A Toeplitz matrix $A = T_n \langle S ; T \rangle$ is walk-ensured if and only if there exists some positive integer $M$ such that $P_i(A) = R_i(A)$ for every integer $i \ge M$.
\end{proposition}

\begin{proof}
By the definitions of $P_i(A)$ and $R_i(A)$, $A$ is walk-ensured if and only if there exists some positive integer $M$ such that $P_i(A) \subseteq R_i(A)$ for any $i \ge M$. 
Yet, by \eqref{eq:pqr}, $P_i(A) \subseteq R_i(A)$ if and only if $P_i(A) = R_i(A)$ for any positive integer $i$. 
Therefore $A$ is walk-ensured if and only if there exists some positive integer $M$ such that $P_i(A) = R_i(A)$ for every integer $i \ge M$. 
\end{proof}

\section{Contractions of digraphs derived from residue classes}~\label{sec:cont}

For a walk-ensured Toeplitz matrix $A$, $P_i(A)=R_i(A)$ by Proposition~\ref{prop:walk-ensured}. Accordingly, the following lemma demonstrates that the directed walks between two vertices exist periodically in terms of their lengths.

\begin{lemma}[\cite{period}]\label{lem:abcd}
For nonempty sets $S, T \subseteq[n-1]$, let $A = T_n \langle S; T \rangle$, $d^+=\gcd(S+T)$, and $d = \gcd (S \cup T)$.
Then the following are true for any positive integer $i$:
\begin{itemize}
\item[(a)] $P_i (A)= P_{i+d^+/d}(A)$;
\item[(b)] $P_i(A),\ldots, P_{i-1+d^+/d}(A)$ are mutually disjoint;
\item[(c)] $P_i(A) = \{\ell \in \mathcal{I}_n \mid \ell-s_1 \in P_{i-1} (A) \text{ or } \ell+t_1 \in P_{i-1}(A)\} $ for any $i \ge 2$
\end{itemize}
where $\mathcal{I}_n$ denotes the set of integers on the interval $[-n+1, n-1]$.
\end{lemma}

In the digraph of a walk-ensured Toeplitz matrix $T_n\langle S;T\rangle$, two vertices belonging to the same congruence classes of modulo $d = \gcd(S \cup T)$ are strongly connected as shown in the following proposition.

\begin{proposition}\label{cor:walk}
Let $D$ be the digraph of a walk-ensured Toeplitz matrix $T_n\langle S;T\rangle$, and $d = \gcd(S \cup T)$.
If $u \equiv v \pmod {d}$, then there is a directed $(u,v)$-walk in $D$.
\end{proposition}

\begin{proof}
Let $A = T_n \langle S; T\rangle$.
There are some integers $p$ and $q$ such that $d = ps_1+qd^+$ by B\'ezout's identity where $d^+ = \gcd(S+T)$.
By the hypothesis, $d \mid v-u$ and so there exists some integer $k$ such that $v-u = k d$.
Since $A$ is walk-ensured, there exists some positive integer $m$ such that $P_i (A) = R_i(A)$ for any $i \ge m$.
Now we have
\begin{align*}
v-u &= kd \\
&= pks_1 + qkd^+ \\
&\equiv pks_1 \pmod{d^+} \\
&\equiv pks_1 + md^+/ds_1 \pmod{d^+}
\end{align*}
and so $v - u \in P_{pk+md^+/d}(A)$ by Lemma~\ref{lem:abcd}(a).
Since $pk+md^+/d \ge m$, $P_{pk+m d^+/d}(A) = R_{pk+m d^+/d}(A)$.
Therefore $v-u \in R_{pk+md^+/d}(A)$ and so there is a directed $(u,v)$-walk.
\end{proof}

Based on Proposition~\ref{cor:walk}, we introduce a contraction of a digraph by using certain residue classes as follows.

\begin{definition}
For a digraph $D$ of order $n$ and a positive integer $d$, let $D/\mathbb{Z}_d$ be the digraph obtained from $D$ by contracting the vertices in the congruence class $\{v \in [n] \mid v \equiv i \pmod d\}$ of $i$ to $i$ for each $i = 1, 2, \ldots, d$, that is, there is an arc $(i,j)$ in $D/\mathbb{Z}_d$ if and only if there is an arc $(m,l)$ in the digraph of $D$ such that $i \equiv m \pmod d$ and $j \equiv l \pmod d$.
\end{definition}

For the digraph $D$ of a Toeplitz matrix with a rather simple structure, $D/\mathbb{Z}_d$ can be easily determined.

\begin{proposition}\label{prop:ddcycle}
Given integers $d,n,s$ with $d \nmid s$ and $s \le n - d$, $D(T_n\langle s; \emptyset \rangle)/\mathbb{Z}_d = D(T_d \langle r;d-r \rangle)$ where $r$ is the remainder when $s$ is divided by $d$.
\end{proposition}

\begin{proof}
Let $r$ be an integer with $r \equiv s \pmod d$ and $0< r < d$.
By definition, there are arcs
\begin{equation}\label{eq:arcs}
(1, 1+s), (2, 2+s), \ldots, (d, d+s)
\end{equation}
in the digraph of $T_n \langle s; \emptyset \rangle$.
Then, for any integer $1 \le i \le d-r$, there is an arc $(i, i+r)$ in $D$.

Now consider an integer $i$ with $d-r+1 \le i \le d$.
Then $i+s \equiv i+r-d \pmod{d}$ and $1 \le i+r-d \le d$. 
Therefore there is an arc $(i, i+r-d)$ in $D$ by \eqref{eq:arcs}.
Therefore the digraph of $T_d \langle r;d-r \rangle$ is a subgraph of $D$.

Take any arc $(i, j)$ in $D$.
Then $j-i \equiv s \pmod d$.
Since $1 \le i \le d$ and $1 \le j \le d$, $-d+1 \le j-i \le d-1$.
Since $0 < r \le d-1$, $j-i = r $ or $r-d$.
Therefore $D$ is a subgraph of the digraph of $T_d \langle r;d-r \rangle$.
Thus $D=T_d \langle r;d-r \rangle$.
\end{proof}

\begin{proposition}\label{prop:cycle}
Let $D$ be the digraph of $T_n \langle s ; n-s \rangle$.
Then $D$ is a disjoint union of directed cycles each of which has the vertex set $\{i, i+d, \ldots, i+(n/\gcd(n,s)-1)d\}$ for some $i \in [d]$.
\end{proposition}

\begin{proof}
 If $v \le n-s$, then $N^+(v) = \{ v+s\}$ and if $v > n-s$, then $N^+(v) = \{v-(n-s)\}$.
Moreover, if $v > s$, then $N^-(v) = \{ v-s \}$ and if $v \le s$, then $N^-(v) = \{ v+ (n-s)\}$.
Therefore every vertex in $D$ has indegree one and outdegree one.
Thus $D$ is a disjoint union of directed cycles.

Let $\gcd(n,s) = d$.
If there is a $(u,v)$-arc, then $v-u \equiv s \pmod{n}$.
Take a cycle $C:=v_1 \rightarrow v_2 \rightarrow \cdots \rightarrow v_k \rightarrow v_1$.
Then $0 = (v_1-v_k)+(v_k-v_{k-1})+\cdots + (v_2-v_1) \equiv ks \pmod{n}$.
Therefore $k \mid n/\gcd(n,s)$.
Moreover, if there is a $(u,v)$-arc, then $u \equiv v \pmod{d}$.
Thus \[\{v_1, \ldots, v_k\} \subseteq \{i, i+d, \ldots, i+(n/\gcd(n,s)-1)d\}\]
for some $i \in [d]$.
Hence \[\{v_1, \ldots, v_k \} = \{i, i+d, \ldots, i+ (n/\gcd(n,s)-1)d \}\] since $k \mid n/\gcd(n,s)$. 
\end{proof}

The following is an immediate consequence of Propositions~\ref{prop:ddcycle} and \ref{prop:cycle}.

\begin{corollary}\label{cor:ddcycle}
Given integers $d,n,s$ with $d \nmid s$ and $s \le n - d$, $D(T_n\langle s; \emptyset \rangle)/\mathbb{Z}_d$ is a disjoint union of directed cycles.
\end{corollary}

In the rest of paper, we will investigate the effect on the lengths of directed walks between two vertices when adding parameter $s^* \in [n-1]$ to either $S$ or $T$ in the Toeplitz matrix $T_n\langle S, T \rangle$.
To this end, we provide a method to determine the existence of directed walks of a specific length between two vertices in $D(T_n\langle S \cup {s^\ast}; T \rangle)$ using a simpler digraph $D(T_n\langle s^\ast; \emptyset \rangle)/\mathbb{Z}_d$.

\begin{lemma}\label{lem:walk1}
Let $T_n\langle S; T \rangle$ be a Toeplitz matrix and $d= \gcd(S\cup T)$.
If there is a directed $(u,v)$-walk with $\ell$ $s^\ast$-arcs in the digraph of $T_n\langle S\cup \{ s^\ast \} ; T \rangle$ for some $s^\ast \in [n-1]$, then there is a directed $(u',v')$-walk $W$ of length $\ell$ in $ D(T_n\langle s^\ast; \emptyset \rangle)/\mathbb{Z}_d$ for some integers $u', v' \in \{1, 2, \ldots, d'\}$ with $u' \equiv u \pmod {d}$ and $v' \equiv v \pmod {d}$.
\end{lemma}

\begin{proof}
Let $D^\ast = T_n\langle S\cup \{s^\ast\};T\rangle$ and $D_d = D(T_n\langle s^\ast; \emptyset \rangle)/\mathbb{Z}_d$.
For any directed $(u_0,v_0)$-walk of length $l_0$ with no $s^\ast$-arc in $D^\ast$,
there are some nonnegative integers $a_i$ and $b_i$ such that
\[
v_0-u_0 = \sum_{i=1}^{k_1} a_is_i - \sum_{i=1}^{k_2} b_it_i.
\]
Since $d\mid s_i$ and $d\mid t_j$ for any $i \in \{1, 2, \ldots, k_1\}$ and $j \in \{1, 2, \ldots, k_2\}$, it follows that
\begin{equation*}
d \mid v_0-u_0.
\end{equation*}

Suppose that there is a directed $(u,v)$-walk $W$ of length $\ell$ in $T_n\langle S\cup \{ s^\ast \} ; T \rangle$.
If there is no $s^\ast$-arc in $W$, then $u \equiv v \pmod {d}$ and so there is a directed $(u',u')$-walk $u'$ where $u'$ is an integer in $\{1,2, \ldots, d\}$ with $u' \equiv u \pmod {d}$.

Now, suppose there is an $s^\ast$-arc in $W$.
Then
\[
W = u_1 \xrightarrow{W_1} v_1 \xrightarrow{s^\ast} u_2 \xrightarrow{W_2} v_2 \xrightarrow{s^\ast} \cdots \xrightarrow{s^\ast} u_{\ell+1} \xrightarrow{W_{\ell+1}} v_{\ell+1}
\]
for some directed walks $W_1, \ldots, W_{\ell+1}$ with no $s^\ast$-arcs and some positive integer $\ell$.
Since there is no $s^\ast$-arc in $W_i$, $v_i \equiv u_i \pmod {d}$ for each $i \in \{1, 2, \ldots, \ell+1\}$.
Therefore there are vertices $w_1, \ldots, w_{\ell+1}$ in $D_d$ satisfying $w_i \equiv u_i \equiv v_i \pmod d$. 
For any integer $i \in \{1, 2, \ldots, \ell\}$, since there is an $s^\ast$-arc $(v_i, u_{i+1})$ in $D^\ast$, there is an arc $(w_i, w_{i+1})$ in $F$.
Therefore there is a directed $(w_1, w_{\ell+1})$-walk of length $\ell$ in $D_d$.
\end{proof}

\begin{lemma}\label{lem:walk2}
Let $T_n\langle S; T \rangle$ be a walk-ensured Toeplitz matrix.
If there is a directed $(u',v')$-walk of length $\ell$ in $D(T_n\langle s^\ast; \emptyset \rangle)/\mathbb{Z}_d$ for $d= \gcd(S\cup T)$ and some $s^\ast \in [n-1]$,
then there is a directed $(u,v)$-walk with $\ell$ $s^\ast$-arcs in the digraph of $T_n\langle S\cup \{s^\ast\} ; T \rangle$ for every vertices $u$ and $v$ with $u \equiv u' \pmod {d}$ and $v \equiv v' \pmod{d}$.
\end{lemma}

\begin{proof}
Let $D^\ast = T_n\langle S\cup \{s^\ast\};T\rangle$ and $D_d = D(T_n\langle s^\ast;\emptyset \rangle)/\mathbb{Z}_d$.
Suppose there is a directed $(u',v')$-walk $W'$ in $D_d$.
Then $W'$ has the following form \[W'=u_0 \rightarrow u_1 \rightarrow \cdots \rightarrow u_{\ell}\] for some nonnegative integer $\ell$ and some vertices $u_0, u_1, \ldots, u_{\ell}$ with $u_0 = u'$ and $u_\ell = v'$.
Take any $i \in \{0, \ldots, \ell-1\}$.
Since there is an arc $(u_i, u_{i+1})$ in $D_d$, there is an $s^\ast$-arc $(v_i, w_{i+1})$ in $D^\ast$ with $v_i \equiv u_i \pmod {d'}$ and $w_{i+1} \equiv u_{i+1} \pmod {d'}$.
By Proposition~\ref{cor:walk}, there is a directed $(u_i,v_i)$-walk and $(w_{i+1},u_{i+1})$-walk in $D^\ast$.
Thus there is a directed $(u_i, u_{i+1})$-walk in $D^\ast$.
Since $i$ was arbitrarily chosen, there is a directed $(u_0,u_{\ell})$-walk in $D^\ast$.
By Proposition~\ref{cor:walk}, there are a directed $(u,u_0)$-walk and $(u_{\ell},v)$-walk and so there is a directed $(u,v)$-walk with $\ell$ $s^\ast$-arcs in $D^\ast$.
\end{proof}

\section{Periods of Toeplitz matrices}\label{sec:pfthm2}

In this section, we provide the periods of Toeplitz matrices. 
Firstly, we compute the periods of walk-ensured Toeplitz matrices. 
Then we propose several methods to find the periods of Toeplitz matrices even if they are not walk-ensured. 

\begin{theorem}\label{thm:walk-ensuredpr}
Let $A=T_n\langle S; T \rangle$ be a walk-ensured Toeplitz matrix.
Then the period of $A$ is $\gcd(S+T)/\gcd(S\cup T)$.
Furthermore, if $\gcd(S+T) \le n$, then  the graph sequence $\{C^m(D(A))\}_{m=1}^\infty$ converges.
\end{theorem}

\begin{proof}
Take any two vertices $u$ and $v$ in $D(A)$.
By Lemma~\ref{lem:cwcorr} and Proposition~\ref{lem:xiperiod}, if there is a directed $(u,v)$-walk of length $\ell$ in $D(A)$, then $v-u \in P_\ell(A)$.
Since $A$ is walk-ensured, $P_\ell(A) = R_\ell(A)$ for a sufficiently large integer $\ell$.
Therefore, if $v-u \in P_\ell(A)$, then there is a directed $(u,v)$-walk of length $\ell$.
Thus
\statement{st:walk}{
there is a directed $(u,v)$-walk of length $\ell$ if and only if $v-u \in P_\ell(A)$.}

Since $(u,v)$-entry of $A^\ell$ is $1$ if and only if there is a directed $(u,v)$-walk of length $\ell$, it is true  by the above observation that $(u,v)$-entry of $A^\ell$ is $1$ if and only if $v-u \in P_\ell(A)$.
Thus the period of $A$ is equal to the period of the sequence $\{P_i(A)\}_{i=1}^\infty$.
By Lemma~\ref{lem:abcd}(a) and (b), $\{P_i(A)\}_{i=1}^\infty$ has period $\gcd(S+T)/\gcd(S\cup T)$.
Hence the period of $A$ is $\gcd(S+T)/\gcd(S\cup T)$.

To show the ``furthermore" part, suppose that
\begin{equation}\label{eq:minsmint}
\gcd(S+T) \le n.
\end{equation}
Take vertices $x,y \in [n]$.
By the definition of $\ell$-step competition graph, $xy \in C^\ell (D(A))$ if and only if there is a vertex $z$ in $[n]$ such that there are a directed $(x,z)$-walk of length $\ell$ and a directed $(y,z)$-walk of length $\ell$ in $D(A)$.
By (S1), $xy \in C^\ell (D(A))$ if and only if $z-x \in P_\ell (A)$ and $z-y \in P_\ell(A)$.

Suppose $x \equiv y \pmod {d^+}$ where $d^+ = \gcd(S+T)$.
By the division algorithm, there exists some integer $w \in [d^+]$ such that $\ell s_1 + x \equiv w \pmod {d^+}$, which implies $w - x \in P_\ell(A)$.
By \eqref{eq:minsmint}, $w \in [d^+]\subseteq [n]$ and so $w$ is a vertex in $D(A)$.
Since $x \equiv y \pmod {d^+}$, $w-y \in P_\ell(A)$.
Therefore we have shown that if $x \equiv y \pmod {d^+}$, then $xy \in C^\ell(D(A))$.

By the definition of $P_\ell(A)$, if $xy \in C^\ell (D(A))$, then $x \equiv y \pmod {d^+}$.
Thus $xy \in C^\ell(D(A))$ if and only if $x \equiv y \pmod {d^+}$.
Eventually, we have shown that, for a sufficiently large $\ell$, the adjacency matrix of $C^\ell(D(A))$ is $T_n \langle d^+, 2d^+, \ldots, \lfloor n/d^+ \rfloor d^+\rangle$.
Hence the graph sequence $\{C^m(D(A))\}_{m=1}^\infty$ converges to $T_n \langle d^+, 2d^+, \ldots, \lfloor n/d^+ \rfloor d^+\rangle$.
\end{proof}

Given square matrices $A= (a_{ij})$ and $B=(b_{ij})$ of order $n$, we write $A \le B$ if $a_{ij} \le b_{ij}$ for every $i, j \in [n]$.
We note that if $A \le B$, then $D(A)$ is a subdigraph of $D(B)$ and so every directed walk in $D(A)$ is also a directed walk in $D(B)$.
Therefore,
\statement{st:ambm}{if $A \le B$, then $A^m \le B^m$. }

\begin{theorem}\label{cor:gcdsame}
Let $A=T_n\langle S; T \rangle$ be a walk-ensured Toeplitz matrix.
Assume that $S^\ast$ and $T^\ast$ are subsets of $[n-1]$ such that $S \subseteq S^\ast$ and $T \subseteq T^\ast$.
If $\gcd(S + T)=\gcd(S^\ast+T^\ast)$, then $T_n \langle S^\ast ; T^\ast \rangle$ has period $\gcd(S +T) / \gcd(S \cup T)$.
\end{theorem}

\begin{proof}
Let $B = T_n \langle S^\ast ; T^\ast \rangle$.
Since $\gcd(S^\ast + T^\ast)=\gcd(S+T)$,
\[
P_i(A) = P_i(B)
\]
 for any positive integer $i$.
By the definition of walk-ensured Toeplitz matrix, there exists some positive integer $M$ such that $P_m(A)=R_m(A)$ for any integer $m \ge M$.
By (S2), since $A \le B$, for any positive integer $m$, $A^m \le B^m$ holds.

Suppose $m \ge M$ and there is a directed $(u,v)$-walk of length $m$ in $D(B)$.
Then $v-u \in P_m(B)$ by Lemma~\ref{lem:cwcorr} and Proposition~\ref{lem:xiperiod}.
Since $P_m(A) = P_m(B)$ and $P_m(A) = R_m(A)$, there exists a directed $(u,v)$-walk of length $m$ in $D(A)$.
Thus, if there exists a directed $(u,v)$-walk of length $m$ in $D(B)$, then there exists a directed $(u,v)$-walk of length $m$ in $D(A)$.
Therefore $A^m \ge B^m$ for every $m \ge M$.
Thus $A^m = B^m$ for every $m \ge M$.
Hence $A$ and $B$ have the same period.
By Theorem~\ref{thm:walk-ensuredpr}, $B$ has period $\gcd(S+T)/\gcd(S\cup T)$.
\end{proof}

\begin{theorem}\label{thm:snksrc}
Let $A = T_n \langle S ; T \rangle$ be a walk-ensured Toeplitz matrix and $B$ be an $n\times n$ Boolean matrix with $A \le B$.
If $D(B-A)/\mathbb{Z}_d$ for $d = \gcd(S \cup T)$ has a source or a sink, then the periods of $A$ and $B$ are the same.
\end{theorem}

\begin{proof}
Let $p$ and $q$ be the periods of $A$ and $B$, respectively.
Suppose that $D(B-A)/\mathbb{Z}_d$ has a sink $u$ and $W$ is a directed $(u,u)$-walk in $D(B)$ in $D(B)$.
We claim that $W$ consists of arcs in $D(A)$.
Suppose, to the contrary, that $xy$ be the first arc on $W$ that is not in $D(A)$.
Since there is a directed $(u,x)$-walk in $D(A)$, $d \mid x-u$.
Let $w$ be a vertex in $D(B-A)/\mathbb{Z}_d$ with $w \equiv y \pmod{d}$.
Then, since $x \rightarrow y$ in $D(B-A)$, $u \rightarrow w$ in  $D(B-A)/\mathbb{Z}_d$, which contradicts to the fact that $u$ is a sink.
Therefore every directed walk from $u$ in $D(B)$ consists of arcs in $D(A)$.
Thus the $(u,u)$-entry of $A^m$ is greater than equal to the $(u,u)$-entry of $B^m$.
By the hypothesis that $A \le B$ and (S2), $A^m \le B^m$ and so the $(u,u)$-entry of $A^m$ and the $(u,u)$-entry of $B^m$ are the same.
Therefore
\statement{st:prsm}{
the period of $\{(A^m)_{uu}\}_{m=1}^\infty$ and the period of $\{(B^m)_{uu}\}_{m=1}^\infty$ are the same.
}

We note that $0 \in P_m(A)$ if and only if $m \equiv 0 \pmod{d^+/d}$ where $d^+ = \gcd(S+T)$ for any positive integer $m$.
Since $A$ is walk-ensured, there exists an integer $M$ such that $P_m(A) = R_m(A)$ for any integer $m \ge M$.
Therefore, for every integer $m \ge M$ with $m \equiv 0 \pmod{d^+/d}$, we have $(A^m)_{uu} =1$.
On the other hand, by Lemma~\ref{lem:cwcorr} and Proposition~\ref{lem:xiperiod}, if $(A^m)_{uu} = 1$ for some positive integer $m$, then $0 \in P_m(A)$ and so $m \equiv 0 \pmod{d^+/d}$.
Thus the period of $\{(A^m)_{uu}\}_{m=1}^\infty$ is $d^+/d$ which equals $p$ by Theorem~\ref{thm:walk-ensuredpr}.
Since $q$ is divided by the period of $\{(B^m)_{uu}\}_{m=1}^\infty$,
$p \mid q$ by (S3).
Therefore $q = kp$ for some positive integer $k$.
Thus, for sufficiently large $m$ and any nonnegative integer $i$,
\begin{equation}\label{eq:kp}
B^m = B^{m+ikp}.
\end{equation}

By Theorem~\ref{lem:pqr}, there exists a $(v,v)$-directed walk of length $\ell p$ in $D(A)$ for any vertex $v$ and every integer $\ell$ with $\ell p \ge M$ for sufficiently large $M$.
Then, since $A \le B$, there exists a $(v,v)$-directed walk of length $\ell p$ in $D(B)$ for any vertex $v$ and every integer $\ell$ with $\ell p\ge M$.
Therefore $B^j \le B^{j+\ell p}$ for any positive integer $j$ and every integer $\ell$ with $\ell p \ge M$.
Then
\[
B^m \le B^{m+\ell p} \le B^{m+2\ell p} \le \cdots \le B^{m+k\ell p} = B^m,
\]
by \eqref{eq:kp}.
Therefore $B^m = B^{m+\ell p}$ for sufficiently large $m$ and any integer $\ell$ with $\ell p \ge M$.
Thus $B$ also has the period $p$.
\end{proof}

\begin{corollary}\label{thm:main3}
Let $T_n\langle S; T \rangle$ be a walk-ensured Toeplitz matrix.
Then, for any integer $s^\ast$ satisfying
\[n-\gcd(S\cup T)<s^\ast < n, \]
the periods of $T_n\langle S; T \rangle$ and $T_n\langle S\cup \{s^\ast\};T \rangle$ are the same.
\end{corollary}

\begin{proof}
Let $d = \gcd(S\cup T)$.
Suppose that $n- \gcd(S\cup T) < s^\ast <n$.
Then $n-s^\ast < d$.
We note that there are exactly $n-s^\ast$ $s^\ast$-arcs in $D(T_n \langle S\cup \{s^\ast\} ; T \rangle)$:
\[(1, 1+s^\ast), (2, 2+s^\ast), \ldots, (n-s^\ast, n). \]
We also observe that \[T_n \langle S\cup \{s^\ast\};T\rangle - T_n\langle S;T \rangle = T_n\langle s^\ast ; \emptyset \rangle.\]
Since $n-s^\ast < d$, there are at most $d-1$ arcs in $D(T_n\langle s^\ast;\emptyset \rangle)/\mathbb{Z}_d$ by definition.
Therefore there is a sink in $D(T_n\langle s^\ast;\emptyset \rangle)/\mathbb{Z}_d$.
Hence $T_n\langle S;T \rangle$ and $T_n\langle S\cup \{s^\ast\};T \rangle$ have the same period by Theorem~\ref{thm:snksrc}.
\end{proof}

\section{Walk-ensured Toeplitz matrices}\label{sec:pfthm1}

In the previous section, we provided a way to compute the period of walk-ensured Toeplitz matrices. 
Yet, it is not easy to identify walk-ensured Toeplitz matrices. 
In this section, we present two sufficient conditions for Toeplitz matrices being walk-ensured which can be easily verified as true or false but significantly relaxes ($\star$) as follows.

\begin{theorem}\label{thm:st}
For any integers $s,t$ with $s+t \le n$ and $\gcd(s,t)=1$, if $s \in S \subseteq [n-1]$ and $t \in T \subseteq [n-1]$, then $T_n\langle S;T \rangle$ is a walk-ensured Toeplitz matrix.
\end{theorem}

\begin{theorem}\label{thm:main1}
Let $S$ and $T$ be subsets of $[n-1]$ for some positive integer $n$ satisfying $s_1+t_1 \le n$, $\max \{s_{k_1},t_{k_2}\} \le n- \gcd(s_1,t_1)$,
Then $T_n\langle S;T\rangle$ is a walk-ensured Toeplitz matrix.
\end{theorem}

\begin{example}
For any positive integers $k, n$ with $2k+1 \le n$, $T_n\langle k, n-k;k+1, n-k-1  \rangle$ is a walk-enusred Toeplitz matrix by Theorem~\ref{thm:st} (or Theorem~\ref{thm:main1}), even if the condition ($\star$) is violated.
In addition to this, any number of Toeplitz matrices that do not satisfy ($\star$) but can be identified as a walk-ensured Toeplitz matrix by Theorem~\ref{thm:st} or \ref{thm:main1} can be constructed.
\end{example}

As a matter of fact, the above two theorems can be easily derived from the following theorem. 

\begin{theorem}\label{thm:1}
Let $T_n \langle S;T \rangle$ be a walk-ensured Toeplitz matrix with $d = \gcd (S \cup T)$.
Then for any positive integer $s^* \le n-d$, $T_n\langle S \cup \{s^\ast\}; T \rangle$ and $T_n\langle S ; T \cup \{s^\ast\} \rangle$ are also walk-ensured Toeplitz matrices.
\end{theorem}

\begin{proof}[Proof of Theorem~\ref{thm:st}]
Let $T_n \langle S ; T \rangle$ be a Toeplitz matrix with elements $s \in S$ and $t\in T$ satisfying $s+t \le n$ and $\gcd(s,t)=1$.
By Theorem~\ref{thm:minmax}, $T_n \langle s; t \rangle$ is a walk-ensured Toeplitz matrix.
Since $\gcd( s, t) = 1$, $d = \gcd (S^\ast \cup T^\ast) = 1$ for any $S^\ast \subseteq S$ and $T^\ast \subseteq T$ containing $s \in S^\ast$ and $t \in T^\ast$.
Thus any element in $S\cup T$ is less than or equal to $n-d$.
Therefore we may apply Theorem~\ref{thm:1} repeatedly until we expand $T_n \langle s; t \rangle$ to $T_n \langle S; T \rangle$ keeping the ``walk-ensured'' property.
\end{proof}

\begin{proof}[Proof of Theorem~\ref{thm:main1}]
By Theorem~\ref{thm:minmax}, $T_n \langle s_1 ; t_1 \rangle$ is a walk-ensured Toeplitz matrix.
Since $\max (S \cup T) \le n- \gcd(\min S, \min T)$, for any $i \in [k_1-1]$ and $j \in [k_2-1]$, $s_{i+1} \le n-\gcd(s_1, \ldots, s_i, t_1, \ldots, t_j)$ and $t_{j+1} \le n-\gcd(s_1, \ldots, s_i, t_1, \ldots, t_j)$.
Thus we may apply Theorem~\ref{thm:1} repeatedly to have a walk-ensured Toeplitz matrix $T_n \langle S; T \rangle$.
\end{proof}

The rest of this section is devoted to proving Theorem~\ref{thm:1}. 
By the symmetry of $T_n\langle S; T \rangle$, it is sufficient to show $T_n \langle S \cup \{s^\ast\}; T \rangle$ is walk-ensured Toeplitz matrix for any positive integer $s^\ast \le n-d$ to prove Theorem~\ref{thm:1}.

The following lemma provides a way of computing $\gcd(S^\ast \cup T)$ and $\gcd(S^\ast+T)$ in terms of $\gcd(S \cup T)$ and $\gcd(S+T)$ where $S^\ast = S \cup \{s^\ast\}$.

\begin{lemma}\label{lem:dchange}
For nonempty subsets $S$ and $T$ of $[n]$ and $s^\ast \in [n]$, let $d = \gcd(S\cup T)$, $d^+ = \gcd(S+T)$, and $S^\ast = S \cup \{s^\ast\}$.
Then \[\gcd(S^\ast \cup T) =\gcd(d, s^\ast -s)\quad \text{and}\quad  \gcd(S^\ast + T) = \gcd(d^+, s^\ast-s) \] for any $s \in S$.
\end{lemma}

\begin{proof}
We note that $S^\ast \cup T = (S\cup T) \cup \{s^\ast\}$.
Then $\gcd(S^\ast \cup T) = \gcd( d , s^\ast)$.
Since $d \mid s$, we have $\gcd (d, s^\ast) = \gcd (d, s^\ast - s)$.
Therefore $\gcd (S^\ast \cup T) = \gcd(d, s^\ast - s)$.

Since $S^\ast + T = (S+T) \cup \{s^\ast + t \mid t \in T\}$,
$\gcd (S^\ast + T) = \gcd(d^+ , \gcd(s^\ast + t \mid t\in T))$.
Since $d^+ \mid s+t$ for any $t \in T$,
\begin{align*}
\gcd(d^+ , \gcd(s^\ast + t \mid t\in T)) &= \gcd(d^+, \gcd(s^\ast + t - (s+t) \mid t\in T))\\ &= \gcd(d^+, s^\ast -s).
\end{align*}
Therefore $\gcd(S^\ast+T) = \gcd(d^+, s^\ast -s)$.
\end{proof}

\begin{lemma}\label{lem:walk3}
Let $T_n \langle S; T \rangle$ be a walk-ensured Toeplitz matrix.
Assume that $u$ and $v$ are vertices in $D(T_n\langle S\cup \{s^\ast \}; T \rangle$ satisfying $v-u \equiv cs^\ast + \alpha s_1 \pmod {d^+}$ for some integers $c \ge 0$ and $\alpha$ where  $d^+ = \gcd(S+T)$, $s^\ast \le n-d$ and $d \nmid s^\ast$ for $d = \gcd(S\cup T)$.
Then there is a directed $(u,v)$-walk of length $\alpha + (d^+/d)m$ in $D(T_n\langle S\cup \{s^\ast \}; T \rangle$ for some positive integer $m$.
\end{lemma}

\begin{proof}
Let $A = T_n\langle S;T$ and $A^\ast = T_n\langle S\cup \{s^\ast \}; T \rangle$.
Since $s^\ast \le n-d$ and $d\nmid s^\ast$ by hypothesis, $D(T_n\langle s^\ast;\emptyset \rangle)/\mathbb{Z}_d$ is a disjoint union of directed cycles by Corollary~\ref{cor:ddcycle}.
Then there exists a directed walk of length $c$ starting from $u$.
Therefore, by Lemma~\ref{lem:walk2}, there exists a directed walk $W_0$ with $c$ $s^\ast$-arcs starting from $u'$ in $D(A^\ast)$ where $u'$ is a vertex with $u' \equiv u \pmod{d}$.
Let $v'$ be the terminus of $W_0$.
Since $u \equiv u' \pmod{d}$, there exists a directed $(u,u')$-walk $W'$ in $D(A)$ by Proposition~\ref{cor:walk}.
Since $D(A)$ is a subdigraph of $D(A^\ast)$, $W'$ is a directed $(u,u')$-walk in $D(A^\ast)$.
Since $u \equiv u' \pmod d$ and $v'-u' \equiv cs^\ast \pmod d$, $v'-u \equiv cs^\ast \pmod d$.
Since $p \equiv cs^\ast \pmod d$,
\begin{align*}
v-v' &= (v-u)-(v'-u)\\
&\equiv cs^\ast - cs^\ast \pmod d \\
&\equiv 0 \pmod d.
\end{align*}
Therefore there exists a directed $(v',v)$-walk $W''$ in $D(A)$ by Proposition~\ref{cor:walk}.
Since $D(A)$ is a subdigraph of $D(A^\ast)$, $W''$ is a directed $(v',v)$-walk in $D(A^\ast)$.
Thus
\[
u \xrightarrow{W'}u' \xrightarrow{W_0} v' \xrightarrow{W''} v
\]
is a directed $(u,v)$-walk with $c$ $s^\ast$-arcs in $D(A^\ast)$.
We denote this directed $(u,v)$-walk by $W$.

Let $\ell$ be the length of $W$, $a_i$ be the number of $s_i$-arcs in $W$, $b_j$ be the number of $t_j$-arcs in $W$ for each $i \in [k_1]$ and $j \in [k_2]$.
Then
\[
v-u = cs^\ast + \sum_{i=1}^{k_1} a_is_i - \sum_{i=1}^{k_2} b_it_i
\]
and $\ell = \sum a_i + \sum b_i + c$.
Since $v-u \equiv c(s^\ast-s_1) + \alpha s_1 \pmod {d^+}$, we have
\begin{align*}
cs^\ast +  \sum_{i=1}^{k_1} a_is_i - \sum_{i=1}^{k_2} b_it_i &= v-u \\
&= p \\
&\equiv cs^\ast + (\alpha -c)s_1.
\end{align*}
Therefore $\sum_{i=1}^{k_1} a_is_i - \sum_{i=1}^{k_2} b_it_i \equiv (\alpha -c)s_1 \pmod {d^+}$.
By Proposition~\ref{lem:xiperiod}, $\sum_{i=1}^{k_1} a_is_i - \sum_{i=1}^{k_2} b_it_i \equiv (\sum a_i + \sum b_i)s_1 \pmod {d^+}$.
Since $\ell = \sum a_i + \sum b_i + c$,
\begin{align*}
(\alpha -c)s_1 &\equiv \sum_{i=1}^{k_1} a_is_i - \sum_{i=1}^{k_2} b_it_i \pmod{d^+} \\
&\equiv (\ell-c) s_1 \pmod{d^+}.
\end{align*}
Then, by Lemma~\ref{lem:abcd},
\begin{equation}\label{eq:div}
\ell - \alpha  = (d^+/d)m
\end{equation}
 for some integer $m$.
\end{proof}

We will prove Theorem~\ref{thm:1} by considering the following two cases:
(i) $d \nmid s^\ast$; (ii) $d\mid s^\ast$.

\subsection{The case $d\nmid s^\ast$}

Let $A = T_n\langle S;T \rangle$, $S^\ast = S \cup \{s^\ast\}$,  $A^\ast = T_n\langle S^\ast ; T \rangle$, and $d^+ = \gcd(S+T)$.
Since $d^\ast = \gcd(s^\ast-s_1, d^+)$ by Lemma~\ref{lem:dchange}, there exist some integers $a$ and $b$ such that $d^\ast = (s^\ast-s_1)a + d^+b$ by B\'{e}zout's Identity.
Fix a positive integer $\alpha$ and take an element $p \in P_\alpha(A^\ast)$ and vertex $v$ in $D(A^\ast)$ with $v-p \in V(D(A^\ast))$.
Then $p \equiv \alpha s_1 \pmod {d^\ast}$ and so $p-\alpha s_1 = kd^\ast$ for some integer $k$.
Therefore $p-\alpha s_1 = kd^\ast = k((s^\ast-s_1)a+ d^+b)$ and so
\[
p\equiv ka(s^\ast-s_1) + \alpha s_1 \pmod{d^+}.
\]
By the division algorithm, there exist some integers $c$ and $q$ such that $ka = qd^+ + c$ and $0 \le c < d^+$.
Thus we have
\[
p \equiv c(s^\ast-s_1)+\alpha s_1 \pmod{d^+}.
\]
Therefore, by Lemma~\ref{lem:walk3},
\statement{wpv}{there exists a directed walk $W_{p,v}$ of length $\alpha + (d^+/d)m_{p,v}$ for some integer $m_{p,v}$. }

Let $M$ be a positive integer such that $P_i(A) = R_i(A)$ for any integer $i \ge M$ and
\[
M^\ast = {\rm max} \{ \alpha + (d^+ / d ) m_{p,v}+M \mid \alpha \in [d^+/d], p \in P_\alpha (A^\ast), v\in V(D(A^\ast)) \text{ with } v-p \in V(D(A^\ast))\}.
\]
Take any positive integer $i \ge M^\ast$ and $p \in P_i(A^\ast)$.
By the division algorithm, there is a positive integer $\alpha \in [d^+/d]$ satisfying $i \equiv \alpha \pmod {d^+/d}$.
Then, by Lemma~\ref{lem:abcd}(a), $P_i(A^\ast) = P_\alpha(A^\ast)$ and so $p \in P_\alpha (A^\ast)$.

To show $p \in R_i(A^\ast)$, it is sufficient to show that there is a directed $(x,y)$-walk of length $i$ for any vertices $x,y$ in $V(D(A^\ast))$ such that $y-x = p$.
Let $y, x$ be vertices in $V(D(A^\ast))$ such that $y-x = p$.
By (S4), there is a directed $(x,y)$-walk $W_{p,x}$ of length $\alpha + (d^+/d)m_{p,x}=:\ell$ for some integer $m_{p,x}$.
Since $i \equiv \alpha \pmod {d^+/d}$, we have $i \equiv \ell \pmod {d^+/d}$.
Then, by Lemma~\ref{lem:abcd}(a), $P_0 (A) = P_{i-\ell}(A)$ and so $0 \in P_{i-\ell}(A)$.
By the definitions of $M^\ast$ and $\ell$, $M^\ast \ge \ell + M$ and so
\begin{align*}
i- \ell &\ge i- (M^\ast - M)\\
&\ge M^\ast - (M^\ast - M)\\
&= M.
\end{align*}
Thus, by the choice of $M$, $P_{i-\ell}(A) = R_{i-\ell}(A)$.
Therefore $0 \in R_{i-\ell}(A)$ and so there exists a directed $(y,y)$-walk $W^\ast$ of length $(i-\ell)$ (note that $y-y= 0$).
Finally, we have a directed $(x,y)$-walk
\[
x \xrightarrow{W_{p,x}} y \xrightarrow{W^\ast} y
\]
of length $i$.
Since $x$ and $y$ were selected arbitrarily among the pairs of vertices with difference $p$, we have $p \in R_i(A^\ast)$.
Since $p$ was arbitrarily chosen, $P_i(A^\ast) \subseteq R_i(A^\ast)$ and so $P_i(A^\ast) = R_i(A^\ast)$.
Therefore $P_i(A^\ast) =R_i(A^\ast)$ for any integer $i \ge M^\ast$.
Hence $A^\ast$ is a walk-ensured Toeplitz matrix.

\subsection{The case $d\mid s^\ast$}

Let $s^\ast$ be a positive integer such that $s^\ast \le n-d$.
Let $A=T_n \langle S;T \rangle$, $S^\ast = S \cup \{s^\ast\}$, $A^\ast = T_n\langle S^\ast ; T \rangle$, $d^+ = \gcd (S+T)$, and $d^\ast = \gcd (S^\ast+T)$.
By Lemma~\ref{lem:dchange},
\[
d^\ast = \gcd (d^+, s^\ast - s_1) \quad \text{and} \quad
\gcd (S^\ast \cup T) = \gcd (d, s^\ast - s_1)
\]
where $s_1 = \min S$.
By the definition of $d$, $d\mid s_1$ and so $d\mid s^\ast -s_1$.
Then $\gcd(S^\ast \cup T) = \gcd (d, s^\ast-s_1) = d$.
Since $d^\ast = \gcd (d^+ , s^\ast -s_1)$, $d^\ast \mid d^+$.
Moreover, by B\'{e}zout's identity, there are some integers $a$ and $b$ such that $d^\ast = ad^+ + b (s^\ast-s_1)$.

Since $A$ is a walk-ensured Toeplitz matrix, there is a positive integer $M$ such that $P_i(A) = R_i(A)$ for every integer $i \ge M$.
Let
\[
M^\ast = d^+ (M+d^++1)
\]
and $\ell$ be a positive integer with $\ell \ge M^\ast$.
We will show that $P_\ell(A^\ast) \subseteq R_\ell(A^\ast)$.

To this end, take $p \in P_\ell(A^\ast)$.
Then $p \equiv \ell s_1 \pmod {d^\ast}$ and so there is an integer $i$ such that $\ell s_1-p = id^\ast$.
Thus $p \equiv \ell s_1 + id^\ast \pmod {d^+}$.
By the division algorithm, there is an integer $c$ satisfying
\begin{equation}\label{eq:c}
0 \le c \le d^+-1 \quad\text{and}\quad ib \equiv c \pmod{d^+}.
\end{equation}
Then
\begin{align}\label{eq:p}
p &\equiv \ell s_1 + id^\ast \pmod{d^+} \nonumber \\
&\equiv \ell s_1 + i(ad^+ + b(s^\ast-s_1)) \pmod{d^+} \nonumber \\
&\equiv (\ell - c) s_1 + cs^\ast \pmod{d^+}.
\end{align}

To show $p \in R_\ell(A^\ast)$, take any vertices $u$ and $v$ satisfying $v-u = p$.
For any $i \in \{0, \ldots, c-1\}$, let $v_i$ be the vertex in $[d]$ satisfying \[v_i \equiv u + is^\ast \pmod {d}.\]
Since $v_i \in [d]$ and $s^\ast \le n-d$, there exists a $s^\ast$-arc $(v_i, v_i +s^\ast)$
and we let
\begin{equation}\label{eq:ui}
u_{i+1} = v_i + s^\ast
\end{equation}
for any $i \in \{0, \ldots, c-1\}$.
In the following, we will take two steps to show that there is a directed $(u,v)$-walk of length $\ell$.
In the first step, we construct a directed $(u_j,v_j)$-walk of length at most $M+d^+$ for $j = 0, 1, \ldots, c-1$.
In the second step, we will construct a directed $(u_c, v_c)$-walk of length $\ell-\ell'$ where $\ell'$ is the length of
\[
u=u_0 \xrightarrow{W_0} v_0 \xrightarrow{s^\ast} u_1 \xrightarrow{W_1} v_1 \xrightarrow{s^\ast} \cdots \xrightarrow{s^\ast} u_c.
\]
Then there is a directed $(u,v)$-walk $W$ of length $\ell$:
\[
u= u_0 \xrightarrow{W_0} v_0 \xrightarrow{s^\ast} u_1 \xrightarrow{W_1} v_1 \xrightarrow{s^\ast} \cdots \xrightarrow{s^\ast} u_c \xrightarrow{W_c} v_c=v.
\]

{\bf Step 1.}
Fix $j \in \{0, \ldots, c-1\}$.
We note that $v_0 \equiv u_0 \pmod d$ and, for $j \ge 1$,
\begin{align*}
v_j - u_j &= v_j - (v_{j-1}+s^\ast) \tag{by \eqref{eq:ui}}\\
&\equiv (u+js^\ast) - (u+(j-1)s^\ast + s^\ast) \pmod{d} \\
&\equiv 0 \pmod{d}.
\end{align*}
Therefore $v_j - u_j = \ell_j d$ for some integer $\ell_j$.
Since $d = \gcd (d^+, s_1)$, by B\'{e}zout's identity, there are some integers $x$ and $y$ such that $d= xd^++ ys_1$, and so
\begin{align*}
v_j-u_j &= \ell_j (xd^+ + ys_1) \\
&\equiv \ell_jys_1 \pmod {d^+}.
\end{align*}
Therefore $v_j-u_j \in P_{\ell_jy}(A)$.
By the division algorithm, there exists an integer $\alpha_j$ satisfying
\begin{equation}\label{eq:alpha}
0< \alpha_j \le \frac{d^+}{d} \quad \text{and} \quad \ell_jy -M \equiv \alpha_j \pmod{d^+/d}.
\end{equation}
Then $P_{\ell_jy} (A) = P_{M+\alpha_j}(A)$ by Lemma~\ref{lem:abcd}(a) and so $v_j-u_j \in P_{M+\alpha_j}(A)$.
Since $P_{M+\alpha_j}(A) = R_{M+\alpha_j}(A)$, there is a directed $(u_j,v_j)$-walk $W_j$ of length $M+\alpha_j$.

{\bf Step 2.}
Since $W_0, \ldots, W_{c-1}$ are directed walks in $A$, we have
\begin{equation}\label{eq:length}
v_j - u_j \equiv (M+\alpha_j)s_1 \pmod {d^+}
\end{equation}
by Lemma~\ref{lem:cwcorr} and Proposition~\ref{lem:xiperiod}.
Then
\begin{align*}
v_c-u_c &= (v_c-u_0) - ((u_c-v_{c-1}) + (v_{c-1}-u_{c-1})+ \cdots + (v_1-u_1) + (u_1-v_0)+(v_0-u_0)) \\
&= (v-u) - \sum_{j=0}^{c-1}(u_{j+1}-v_j) - \sum_{j=0}^{c-1} (v_j-u_j) \tag{$\because$ $u_0=u$, $v_c=v$} \\
&= p - cs^\ast -\sum_{j=0}^{c-1} (v_j-u_j) \tag{$\because$ $v-u=p$, \eqref{eq:ui}} \\
&\equiv p-cs^\ast - \sum_{j=0}^{c-1}(M+\alpha_j)s_1 \pmod{d^+} \tag{by \eqref{eq:length}}\\
&\equiv (\ell-c)s_1 + cs^\ast - cs^\ast - \sum_{j=0}^{c-1}(M+\alpha_j)s_1 \pmod{d^+} \tag{by \eqref{eq:p}}\\
&\equiv (\ell-c - \sum_{j=0}^{c-1}(M+\alpha_j)) s_1 \pmod {d^+}
\end{align*}
and so $v_c-u_c \in P_{\ell-c-\ell^\ast}(A)$ where $\ell^\ast = \sum_{j=0}^{c-1} (M+\alpha_j)$.
Since $\ell \ge M^\ast=d^+(M+d^++1)$, we have
\begin{align*}
\ell-c-\ell^\ast &\ge d^+(M+d^++1) - c - \sum_{j=0}^{c-1}(M+\alpha_j) \\
&\ge d^+(M+d^++1) -c - \sum_{j=0}^{c-1}(M+d^+/d) \tag{by \eqref{eq:alpha}} \\
&= d^+(M+d^++1) -c(M+d^+/d+1) \\
&\ge d^+(M+d^++1)-(d^+-1)(M+d^+/d+1) \tag{by \eqref{eq:c}} \\
&\ge M+d^++1 \\
&\ge M.
\end{align*}
Therefore $P_{\ell-c-\ell^\ast}(A) = R_{\ell-c-\ell^\ast}(A)$ by the choice of $M$ and so $v_c-u_c \in R_{\ell-c-\ell^\ast}(A)$.
Thus there is a directed $(u_c,v_c)$-walk $W_c$ of length $\ell-c-\ell^\ast$.
Hence the existence of a directed $(u,v)$-walk $W$ of length $\ell$ is verified.
Since $u$ and $v$ were selected arbitrarily among the pairs of vertices with difference $p$, we have $p \in R_\ell(A^\ast)$.
Since $p$ was arbitrarily chosen, $P_\ell(A^\ast) \subseteq R_\ell(A^\ast)$.
Therefore $P_i(A^\ast) =R_i(A^\ast)$ for any integer $i \ge M^\ast$.
Hence $A^\ast$ is a walk-ensured Toeplitz matrix.

\section{Acknowledgement}
This work was partially supported by Science Research Center Program through the National Research Foundation of Korea(NRF) Grant funded by the Korean Government (MSIP)(NRF-2016R1A5A1008055). G.-S. Cheon was partially supported by the NRF-2019R1A2C1007518. Bumtle Kang was partially supported by the NRF-2021R1C1C2014187. S.-R. Kim and H. Ryu were partially supported by the Korea government (MSIP) (NRF-2017R1E1A1A03070489 and NRF-2022R1A2C1009648).


\begin{thebibliography}{20}

\bibitem{exp} G.-S.Cheon, J.-H.Jung, B. Kang, and S.-R. Kim, \textit{Exponents of primitive directed Toeplitz graphs}, {\em Linear and Multilinear Algebra}, Published online: 21 May 2020.

\bibitem{Bru} R.A. Brualdi, H.J. Ryser, \textit{Combinatorial Matrix Theory}, Cambridge
University Press, 1991.

\bibitem{cho2013competition}
H.~H. Cho and H.~K. Kim.
\newblock The competition index of a nearly reducible boolean matrix.
\newblock {\em Bull. Korean Math. Soc.}, 50(6):2001-2011, 2013.

\bibitem{cho2011competition}
H.~H. Cho and H.~K. Kim.
\newblock Competition indices of strongly connected digraphs.
\newblock {\em Bull. Korean Math. Soc.}, 48(3):637-646, 2011.

\bibitem{kim2008competition}
H.~K. Kim.
\newblock Competition indices of tournaments.
\newblock {\em Bull. Korean Math. Soc.}, 45(2):385-396, 2008.

\bibitem{kim2010generalized}
H.~K. Kim.
\newblock Generalized competition index of a primitive digraph.
\newblock {\em Linear algebra and its applications}, 433(1):72-79, 2010.

\bibitem{kim2015characterization}
H.~K. Kim.
\newblock Characterization of irreducible boolean matrices with the largest generalized competition index.
\newblock {\em Linear Algebra Appl.}, 466:218-232, 2015.

\bibitem{kim2012bound}
H.~K. Kim and S.~G. Park.
\newblock A bound of generalized competition index of a prmitive digraph.
\newblock {\em Linear algebra and its applications}, 436(1):86-98, 2012.

\bibitem{compToep}
G.~S. Cheon, B. Kang, S.-R. Kim, and H. Ryu.
\newblock Row graphs of Toeplitz matrices.
\newblock {\em under review.}

\bibitem{mstepcomp}
H. H. Cho, S.-R. Kim, and Y. Nam.
\newblock The $m$-step competition graph of a digraph.
\newblock {\em Discrete Applied Mathematics}, 105(1-3):115-127, 2000.

\bibitem{period}
G.~S. Cheon, B. Kang, S.-R. Kim, and H. Ryu.
\newblock Matrix periods and competition periods of Boolean Toeplitz matrices.
\newblock {\em Linear Algebra Appl.}, 672(1):228-250, 2023.

\end{thebibliography}
\end{document}